\documentclass[12pt]{article}

\usepackage[a4paper, total={6.2in, 9in}]{geometry}

\usepackage{amsmath}
\usepackage{hyperref}
\usepackage{amssymb}
\usepackage{color}

\usepackage{amsthm}

\theoremstyle{plain}

\newtheorem{theorem}{Theorem}
\newtheorem{proposition}{Proposition}
\newtheorem{lemma}{Lemma}

\newtheorem*{claim}{Claim}

\theoremstyle{remark}
\newtheorem{remark}{Remark}

\newcommand{\E}{\mathbb{E}}

\newcommand{\R}{\mathbb{R}}
\newcommand{\Z}{\mathbb{Z}}
\newcommand{\p}{\mathbb{P}}
\usepackage{bbm}
\usepackage{graphicx}
\usepackage{color}
\usepackage{caption}
\usepackage{subcaption}

\usepackage{mathtools}
\mathtoolsset{showonlyrefs}

\title{Strict inequalities for arm exponents in planar percolation}
\author{Ritvik \textsc{Ramanan Radhakrishnan}\thanks{ritvik.radhakrishnan@math.ethz.ch}~ and Vincent \textsc{Tassion}\setcounter{footnote}{2}\thanks{vincent.tassion@math.ethz.ch}}
\date{\today}

\begin{document}
\maketitle

\begin{abstract}
  We discuss a general method to obtain quantitative improvements of correlation inequalities and apply it to arm estimates for Bernoulli bond percolation on $\Z^2$. Our first result is that the two-arm exponent is strictly larger than twice the one-arm exponent and can be seen as a quantitative improvement of the Harris-FKG inequality. This answers a question of Garban and Steif \cite[Open Problem 13.6]{MR3468568}, which  was motivated by the study of exceptional times in dynamical percolation \cite[section 9]{MR2630053}.  Our second result  is that the monochromatic arm exponents are strictly larger than their polychromatic versions, and can be seen as a quantitative improvement of Reimer's main lemma~\cite[Lemma 4.1]{MR1703130}.  This second result  is not new; it was already proved by  Beffara and Nolin \cite{MR2857240} using a different argument.
\end{abstract}

\section*{Introduction}

We consider critical face percolation on the two dimensional hexagonal lattice. Let $\mathbb{H}$ be the set of faces of the hexagonal lattice, where a hexagon is centred at each point $k+\ell e^{\mathrm i\pi/3}$, for $k,\ell\in \mathbb Z/10$.

We colour each face in $\mathbb{H}$ black or white independently with probability $1/2$. In this colouring, we view hexagons as closed subsets of the plane. In particular, hexagons overlap at their boundaries and some points are coloured both black and white.

\begin{figure}[htbp]
  \centering
 \hfill
    \begin{minipage}[c]{.28\linewidth}
    \includegraphics[width=\textwidth]{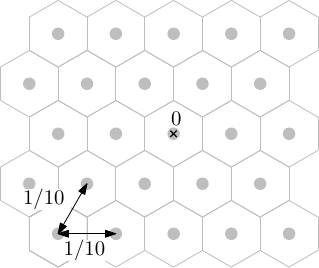}
  \end{minipage}\hfill
    \begin{minipage}[c]{.28\linewidth}
      \includegraphics[width=\textwidth]{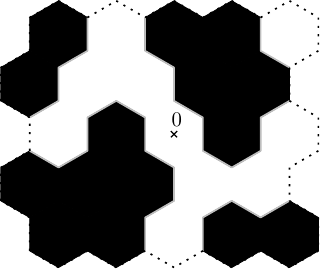}
  \end{minipage}\hfill{}
  \caption{Left: A portion of the plane covered by the hexagons in $\mathbb H$. Right: Percolation on $\mathbb H$, the grey points at the interface between black and white hexagons are formally considered both black and white. }
\end{figure}

For $n\ge 1$, let $A_n(b)$ (resp.\@ $A_n(w)$) be the event that there is a black (resp.\@ white) path connecting the hexagon centred at $0$ to the boundary of $[-n,n]^2$. We emphasize that in the definition of $A_n(c)$, the connection of colour $c$ is from the \emph{hexagon centered at $0$}. In particular, $A_n(c)$ does not imply that the hexagon centred at $0$ is of colour $c$ nor does it imply that the intersection  $A_n(b)\cap A_n(w)$ is non empty. Since $A_n(b)$ is increasing in black and $A_n(w)$ is decreasing, the Harris-FKG inequality (see \cite[Chapter 2]{MR1796851} for instance) implies that
\begin{equation}
\p[{A}_n(b)\cap A_n(w)]\leq\p[A_n(b)]\p[{A}_n(w)].
\end{equation}
Our first result is the following  polynomial improvement of this inequality.

\begin{theorem}\label{thm:polyarm}
     There exists $c>0$ such that for all $n\geq 10$ we have
     
\begin{equation} 
\p[A_n(b)\cap A_n(w)]\leq n^{-c}\cdot \p[A_n(b)]\p[{A}_n(w)].\label{eq:1}
\end{equation}
\end{theorem}

\begin{remark}
  Note that
\begin{equation}
    \mathbb P[A_n(b)]=\mathbb P[A_n(w)],
\end{equation}
since the set of black hexagons has the same law as the set of white hexagons. This colour-switching symmetry shows that~\eqref{eq:1} is equivalent to
\begin{equation}\label{eq:2}
\p[A_n(b)\cap A_n(w)]\leq n^{-c}\cdot \p[A_n(b)]^2.
\end{equation}

\end{remark}

\begin{remark}
As stated, Theorem~\ref{thm:polyarm} is not new. Indeed, for face percolation on hexagons the critical exponents are known: Lawler, Schramm and Werner~\cite{MR1887622}, and Smirnov and Werner~\cite{MR1879816}  showed that
  \begin{equation}
    \label{eq:3}
    \mathbb P[A_n(b)]=n^{-5/48+o(1)}\quad\text{and}\quad\mathbb P[A_n(b)\cap A_n(w)]=  n^{-1/4+o(1)}
  \end{equation}
  as $n\rightarrow \infty$, which directly implies Theorem~\ref{thm:polyarm}. Our main contribution is the proof of this result without using the knowledge of the exponents. In particular, our method also applies to critical bond percolation on the square lattice, where the analogue of Theorem~1 is new. We chose to work with percolation on the hexagonal lattice because it is the natural setup for our second result (Theorem~\ref{thm:monoarm}), whose proof only applies for percolation on the hexagonal lattice.
\end{remark}

\begin{remark}
  The analogues of~\eqref{eq:3} for three or more arms are easier to prove via exploration and Russo-Seymour-Welsh estimates. For more details see \cite[Appendix A]{MR2736153}, which shows that the five arm exponent is is strictly larger than the sum of the four and one arm exponents.
\end{remark} 

\begin{remark}
  In a parallel work, Manolescu and Gassmann~\cite{GassMan} use exploration to prove Theorem~\ref{thm:polyarm} above for the square lattice. Their argument is different from ours and extends naturally to FK percolation, while our argument relies on noise sensitivity results (currently) restricted to Bernoulli percolation. On the other hand, we expect that the method developed in the present paper can be applied to more general events than the one-arm event since we rely on abstract results for Boolean functions.
  
\end{remark}
In their proof of the existence of exceptional times on the  hexagonal lattice using randomised algorithms, Schramm and Steif \cite{MR2630053} used ~\eqref{thm:polyarm} as an input. Their proof did not extend to bond percolation on the square lattice because~\eqref{eq:2} was not known in this setting. Our result shows that the Schramm-Steif method also applies to critical bond percolation on the square lattice.

The existence of exceptional times for critical bond percolation on the square lattice is not new. This result was obtained by Garban, Pete and Schramm \cite{MR2736153} using a different method.

Our second result is related to Reimer's main lemma~\cite[Lemma 4.1]{MR1703130}, which  asserts that for all events $A$, $B$
\begin{equation}\label{eq:4}
    \mathbb P[A\circ B]\le \mathbb P[A\cap \overline B],
\end{equation}
where $A\circ B$ denotes the event that $A$ and $B$ occur disjointly, and $\overline B$ is the image of $B$ through the mapping  $\omega\mapsto \overline \omega$ that swaps the colours white and black (see Section~\ref{sec:quant-reim-lemma} for more precise definitions). Applying Reimer's lemma, we get
\begin{equation}
    \mathbb P[A_n(b)\circ A_n(b)]\le \mathbb P[A_n(b)\cap A_n(w)].
\end{equation}
Beffara and Nolin \cite{MR2857240} proved the following polynomial improvement on the inequality above.

\begin{theorem}\label{thm:monoarm}
     There exists $c>0$ such that for all $n\geq 10$ we have $$\p[{A}_n(b)\circ A_n(b)]\leq n^{-c}\cdot \p[{A}_n(w)\cap A_n(b))].$$
   \end{theorem}

We give a new proof of this result that is conceptually different. Beffara and Nolin  used the standard  Reimer's lemma as in Equation~\eqref{eq:4} but applied  it to different arm events, where  arms ``wind a lot'' around the origin. The polynomial improvement in their approach comes from a clever comparison between probability of arm events and their winding versions. In our approach, we work with standard arm events  but rely on a new  quantitative version of Reimer's lemma.  

\begin{remark}
Beffara and Nolin \cite[Theorem 5]{MR2857240} actually proved a more general version of Theorem~\ref{thm:monoarm} for arbitrarily many monochromatic arms. Our proof also applies to this statement with slight modification but for simplicity we chose to present the proof for two arms only. 
\end{remark}

\paragraph{Acknowledgment} This project was initiated with Christophe Garban, and we owe him a lot. He told us about the two-arm versus one-arm squared problem and its original motivation. We particularly thank him for several discussions, which lead to  key arguments at the early stage of the project. We also thank Hugo Vanneuville for discussions and his encouragement to apply the method to monochromatic arm probabilities. 

 This project has received funding from the European Research Council (ERC) under the European Union’s Horizon 2020 research and innovation programme (grant agreement No 851565).

\section{Quantitative correlation inequalities}
\label{sec:quant-corr-ineq}

In this section, we provide a unified proof of quantitative versions of the Harris-FKG inequality and Reimer's lemma, via interpolation. Interpolation has been used extensively to study  Harris-FKG inequality (see references below).  A novelty in our presentation is the use of a two variable function, which allows us to simultaneously treat both Harris-FKG inequality and Reimer's lemma.

The interpolation formulas we obtain in this section follow from general formulas about Markov semi-groups, see for instance \cite[Section 1.7.1]{MR3155209}. We also mention that in \cite{MR3534069} the authors use interpolation via noise to study quantitative versions of the Harris-FKG inequality specifically. The work \cite{MR4574826} uses this formula in percolation to give another proof of sharp noise sensitivity. 

The study of quantitative versions of the Harris-FKG inequality goes back to Talagrand~\cite{MR1401897}, in which the author derives an error term in the Harris-FKG inequality for general events on the Boolean cube. In~\cite{MR3534069}, the authors improve the error term for events satisfying certain symmetries, and in~\cite{MR3269987}, the authors study analogues of Talagrand's work in Gaussian space. These three papers study various quantitative versions of the FKG inequality for generic events of the Boolean cube. The emphasis in this paper is on two dimensional percolation, where the error term can be interpreted geometrically.
\subsection{Notation}
\label{sec:notation}

Let $n\in \mathbb N$. We use the notation $[n]$ for the set $\{1,\ldots, n\}$.  We define a configuration to be an element $x\in\{0,1\}^n$ and for $i\in [n]$ we write $x(i)$ for the  $i$th coordinate of $x$. We let $x^i$ (resp.\@ $x_i$) be the configuration obtained by replacing the $i$th coordinate of $x$  with $1$ (resp. with $0$).   We write $\overline x$ for the configuration defined by $\overline x(i)=1- x(i)$ for every $i\in[n]$.  For $A\subset \{0,1\}^n$, we  set
\begin{equation}
    \overline A=\{\overline x: x\in A\}.
\end{equation} 
Let $f:\{0,1\}^n\rightarrow \R$ be a function. For each $i\in [n]$, we define the operator $\nabla_i$ by $$\nabla_i f(x)=f(x^{i})-f(x_{i}).$$ We say that $f$ is increasing if for all $x,\,y\in \{0,1\}^n$ we have the implication
\begin{equation}
  \label{eq:5}
  \big(\forall i\quad  x(i)\le y(i)\big ) \implies \big(f(x)\le f(y)\big),
\end{equation}
and we say $f$ is decreasing if $-f$ is increasing. We say that an  event $A$ is increasing (resp.\@ decreasing) if its indicator function $\mathrm 1_A$ is increasing (resp.\@ decreasing).

We equip the hypercube $\{0,1\}^n$ with the uniform probability measure $\mathrm P$ and write~$\mathrm E$ for the corresponding expectation. Fix $t\in [0,1]$. We consider two coupled random configurations $\omega$ and $\omega_t$ defined as follows. Let $\omega$ be uniformly distributed in $\{0,1\}^n$. Independently, let $\eta=(\eta(i))_{i\in[n]}$ be an iid vector, where $\eta(i)$ is a Bernoulli random variable with parameter $t$ and set
  \begin{equation}
        \omega_t(i)=
    \begin{cases}
      \omega(i)&\eta(i) = 0\\
      1-\omega(i) &\eta(i)=1.
    \end{cases}
  \end{equation}
  In other words, $\omega_t$ is obtained by flipping each coordinate of $\omega$ independently with probability $t$. The configuration $\omega_t$ allows us to continuously interpolate between $\omega$ (when $t=0$), an independent copy (when $t=1/2$), and $\overline\omega$ (when $t=1$). We will repeatedly use the following elementary facts.
  
  \begin{lemma} 
    \label{lem:1}
    \phantom{ }
  
    \begin{enumerate}
    \item For every $t\in [0,1]$, $\omega_t$ has distribution $\mathrm P$.
      
    \item The two configurations  $\omega$ and $\omega_{1/2}$ are independent.
    \end{enumerate}

\end{lemma}
  \begin{proof}Let $t\in [0,1]$.
    Since both $\omega$ and $\eta$ are iid vectors, $\omega_t$ is also an iid vector and for every $i$, we have
    \begin{equation}
            \mathbb P[\omega_t(i)=1]=\mathbb P[\omega(i)=1,\eta(i)=0]+\mathbb P[\omega(i)=0,\eta(i)=1]=\frac12 \ t+\frac12\ (1-t)=\frac12.
    \end{equation}
    This proves the first item.  The second item follows from the following computation. When $t=1/2$, for every $i\in[n]$, we have
    \begin{equation}
      \mathbb P[\omega(i)=1,\omega_{1/2}(i)=1]=\mathbb P[\omega(i)=1,\eta(i)=0]=\frac14=\mathbb P[\omega(i)=1]\mathbb [\omega_{1/2}(i)=1].
    \end{equation}
     \end{proof}

\subsection{Quantitative Harris-FKG inequality}

\begin{proposition}
  \label{prop:1}
  Let $f, g:\{0,1\}^n\to \mathbb R_+$ be two non-negative functions with \mbox{$\mathrm E[f],\mathrm E[g]>0$}. Then we have
  
  \begin{equation}
      \mathrm E[f g] =\mathrm  E[f]\mathrm E[g]e^{I},\quad\text{where }I=      \frac 12 \int_0^{1/2} \sum_{i=1}^n\frac{\mathbb E[\nabla_if(\omega)\nabla_ig(\omega_t)]}{\mathbb E[f(\omega)g(\omega_t)]}\; dt. 
    \end{equation}

  \end{proposition}

Note that Proposition~\ref{prop:1} implies the Harris-FKG inequality because $\nabla_if\ge 0$ and $\nabla_ig\ge 0$ when $f,g$ are increasing. We prove Proposition~\ref{prop:1} in Section~\ref{sec:unified-proof-harris}, where we establish the  differential equation
  \begin{equation}
     \frac{d}{dt} \mathbb E[f(\omega)g(\omega_t)]=-\frac12 \sum_{i=1}^n\mathbb E[\nabla_if(\omega)\nabla_ig(\omega_t)].\label{eq:6}
  \end{equation}
We then obtain the equation displayed in Proposition~\ref{prop:1} by integration.

  \begin{remark}\label{rem:4}
When $f=g$ and $t\in [0,1/2]$,  it is possible to prove that the right-hand-side of~\eqref{eq:6} is non-positive   (see the second part of the proof of~\cite[Lemma 3.4]{MR4574826} for example), which implies that $\E[f(\omega)f(\omega_t)]$ is non-increasing for $t\in [0,1/2]$. 
  \end{remark}

\subsection{Quantitative Reimer's lemma}
\label{sec:quant-reim-lemma}

Let $A\subset \{0,1\}^n$ be an event, let $x\in \{0,1\}^n$ and let $I\subset [n]$. We say $(x,I)$ is a witness for $A$ if $\{y\in\{0,1\}^n:y\vert_ I=x\vert_I\}\subset A$. Let $B\subset \{0,1\}^n$ be another event and let $y\in\{0,1\}^n$. We say $(A,B)$ occurs disjointly on $(x,y)$ if there exist disjoint subsets $I,J\subset [n]$ such that $(x,I)$ is a witness for $A$ and $(y,J)$ is a witness for $B$. This generalises the standard notion of disjoint occurrence, in the sense that
\begin{equation}
     \text{$(A,B)$ occurs disjointly on $(x,x)$} \iff x\in A\circ B, 
\end{equation}
for every configuration $x$. We will also use the following property. 
\begin{lemma}
  \label{lem:2}
   For all increasing events  $A,B$ and for all $x\in\{0,1\}^n$ we have
  \begin{equation}
        \text{$(A,B)$ occurs disjointly on $(x,\overline x)$}\iff  x\in A\cap \overline B.
  \end{equation}
  \end{lemma}
  \begin{proof}
    If $(A,B)$ occurs disjointly on $(x,\overline x)$, then $x\in A$ and $\overline x\in B$, that is, $x\in A\cap \overline B$. For the other implication, assume  that $x\in A\cap \overline B$. Define $I=\{i\in[n]\: : \: x(i)=1\}$ and $J=I^c$. Since $A$ is increasing, $(x,I)$ is a witness for $A$. Similarly, since $\overline B$ is decreasing, $(x,J)$ is a witness for $\overline B$, or equivalently, \@ $(\overline x,J)$ is a witness for $B$.   Therefore, $(A,B)$ occurs disjointly in $(x,\overline x)$. 
  \end{proof}

Given a function $F:\{0,1\}^n\times\{0,1\}^n\to \mathbb R$, we define
\begin{equation}\label{eq:defD_i}
   D_iF(x,y)=\left[F(x^{i},y_{i})-F(x_{i},y_{i})\right]\cdot\left[ F(x_{i},y^{i})-F(x_{i},y_{i})\right].
\end{equation}

\begin{proposition}\label{prop:2}
    Let $A, B$ be two increasing events such that $\mathrm P[A\circ B]>0$. For each $x,y\in \{0,1\}^n$, let $F(x,y)$ be the indicator  function that 
  $(A,B)$ occurs disjointly in $(x,y)$. Then we have
  \begin{equation}
    \mathrm P[A\circ B]=\mathrm P[A\cap \overline B] e^{-J},\quad\text{ where } J= \frac12 \int_0^{1} \sum_{i=1}^n\frac{\mathbb E[D_iF(\omega,\omega_t)]}{\mathbb E[F(\omega,\omega_t)]} \; dt.
  \end{equation}
\end{proposition}
Since $F$ is increasing, we have $D_iF\ge 0$, which implies that $J\ge 0$. In particular, the equation above directly implies Reimer's Lemma for increasing events.

We prove Proposition~\ref{prop:2} in Section~\ref{sec:unified-proof-harris} where we establish
\begin{equation}
  \label{eq:7}
  \frac d {dt} \mathbb E[F(\omega,\omega_t)]=-\frac12\sum_{i=1}^n\mathbb E[D_iF(\omega,\omega_t)].
\end{equation}

\subsection{A general  interpolation formula}

For all
$F:\{0,1\}^n\times\{0,1\}^n\to \mathbb R
$ and all  $i\in [n]$,  we define  $\nabla_{ii}F:\{0,1\}^n\times\{0,1\}^n\to\{0,1\}$  by setting 
\begin{equation}\label{eq:8}
\nabla_{ii}F(x,y)=F(x^{i},y^{i})+F(x_{i},y_{i})-F(x^{i},y_{i})-F(x_{i},y^{i})
\end{equation}
for every $x,y\in \{0,1\}^n$.
The operator $\nabla_{ii}$ can be understood as a second-order discrete partial derivative. Indeed, one can check that $$\nabla_{ii} =\nabla_i^1 \circ \nabla_i^2,$$
where $\nabla^1_i$ and $\nabla^2_i$  are defined by  $\nabla^1_i F(x,y)=F(x^{i},y)-F(x_{i},y),$ and $\nabla^2_i F(x,y)=F(x,y^{i})-F(x,y_{i}).$  

\begin{lemma}\label{lem:mainformula}
For every function $F:\{0,1\}^n\times \{0,1\}^n\rightarrow \R$, we have \begin{equation}\label{eq:9}
        \frac{d}{dt}\E[F(\omega, \omega_t)]=-\frac{1}{2}\sum_{i=1}^{n}{\E[\nabla_{ii} F(\omega, \omega_t)]}.
    \end{equation}
    
\end{lemma}

\begin{proof}
  For every $t_1,\ldots,t_n\in[0,1]$, let  $\xi$ be a  configuration obtained  from $\omega$ by flipping the $i$-th coordinate   with probability $t_i$ (independently of $\omega(i)$ and the other coordinates) and  define  $\psi:[0,1]^n\rightarrow \R$ by $$\psi(t_1,\ldots, t_n)=\E[F(\omega, \xi)].$$  With this notation, we have $\psi(t,\ldots,t)=\E[F(\omega, \omega_t)]$ and by the chain rule we obtain 
  \begin{equation}
    \label{eq:10}
    \frac{d}{dt}\E[F(\omega, \omega_t)] = \sum_{i=1}^n \frac{\partial \psi}{\partial t_i}(t,\ldots,t).
\end{equation}
Fix $i\in[n]$. By summing  over the possible values  of $\omega(i)$ and $\xi(i)$, for every $t_1,\ldots, t_n$ we can  expand $\psi(t_1,\ldots, t_n)$ as
\begin{equation}
\frac{1-t_i}{2}\big(\E[F(\omega^{i},\xi^{i})]+\E[F(\omega_{i},\xi_{i})]\big)+\frac{t_i}{2}\big(\E[F(\omega^{i},\xi_{i})]+\E[F(\omega_{i},\xi^{i})]\big).
\end{equation}
Since each of the four expectation terms does not depend on $t_i$ we obtain
\begin{align}
  \frac{\partial \psi}{\partial t_i}(t_1,\ldots,t_n)&=-\frac{1}{2}\left(\E[F(\omega^{i},\xi^{i})]+\E[F(\omega_{i},\xi_{i})]-\E[F(\omega^{i},\xi_{i})]-\E[F(\omega_{i},\xi^{i})]\right)\\
  &=-\frac{1}{2}\E[\nabla_{ii} F(\omega, \xi)].
\end{align}
Applying this identity to $(t_1,\ldots,t_n)=(t,\ldots,t)$ we obtain
$$\frac{\partial \psi}{\partial t_i}(t,\ldots,t)=-\frac{1}{2}\E[\nabla_{ii} F(\omega, \omega_t)],$$ which, together with  \eqref{eq:10}, completes the proof.
\end{proof}

\subsection{A unified proof for Harris inequality  and Reimer's lemma}
\label{sec:unified-proof-harris}
We prove both Proposition~\ref{prop:1} and Proposition~\ref{prop:2}  using a ``duplication'' principle. For Proposition~\ref{prop:1}, we consider a continuous interpolation between $(\omega,\omega)$ and $(\omega,\xi)$, where  $\xi$ is an independent copy of $\omega$. For Proposition~\ref{prop:2}, we interpolate between $(\omega,\omega)$ and $(\omega,\overline \omega)$. In each case, we use Lemma~\ref{lem:mainformula} to control the microscopic variation in the continuous interpolation. 

\begin{proof}[Proof of Proposition~\ref{prop:1}]
 Let $f,g:\{0,1\}^n\to \mathbb R_+$ satisfying $\mathrm{E}[f], \mathrm{E}[g]>0$ and define \mbox{$\phi:[0,1/2]\to\mathbb R$} by 
  \begin{equation}
      \phi(t)=\mathbb E[f(\omega)g(\omega_t)].
  \end{equation}
The assumptions that $f,g$ are non-negative and $\mathrm{E}[f],\mathrm{E}[g]>0$ imply that $\phi(t)>0$ for every $t\in(0,1/2]$. By integrating $\frac d{dt}(\log \phi(t))=\phi'(t)/\phi(t)$ between $0$ and $1/2$ we obtain
\begin{equation}
  \label{eq:11} \phi(0)=\phi(1/2)\exp\left(-\int_0^{1/2}\frac{\phi'(t)}{\phi(t)}\, dt \right).
\end{equation}
  By definition, we have
  \begin{equation}
        \phi(0)=\mathbb E[f(\omega)g(\omega)]=\mathrm E[fg].
  \end{equation}
  By Lemma~\ref{lem:1}, $\omega_{1/2}$ is an independent copy of $\omega$, hence
  \begin{equation}
        \phi(1/2)= \mathbb E [f(\omega)] \mathbb E[g(\omega_{1/2})]=\mathrm E[f] \mathrm E[g].
  \end{equation}
  By applying   Lemma~\ref{lem:mainformula} to  $F(x,y)=f(x)g(y)$ and using  $\nabla_{ii}F(x,y)=\nabla_i f(x)\nabla_ig(y)$, we obtain 
  \begin{equation}
       \forall t\in [0,1/2]\quad \phi'(t)=-\frac12 \sum_{i\in[n]}\mathbb E[\nabla_if(\omega)\nabla_i g(\omega_t)].
  \end{equation}
Plugging the three displayed equations above in Equation~\eqref{eq:11}, we obtained the desired identity.
    \end{proof}

\begin{proof}[Proof of Proposition~\ref{prop:2}]
Let $A,B$ be two increasing events, and $F$ as in the statement of the proposition. Define $\psi:[0,1]\to \mathbb R$ by
\begin{equation}
 \psi(t)=\mathbb E[F(\omega,\omega_t)]. 
\end{equation}
The assumption $\mathrm P[A\circ B]>0$ implies that $\psi(t)>0$ for every $t\in[0,1]$.
By integrating $\frac d{dt}(\log \psi(t))=\psi'(t)/\psi(t)$ between $0$ and $1$, we obtain
  \begin{equation}
    \psi(0)=\psi(1)\exp\left(-\int_0^{1}\frac{\psi'(t)}{\psi(t)}\, dt \right).
  \end{equation}
  By definition,  we have
  \begin{equation}
       \psi(0)=\mathbb E[F(\omega,\omega)]=\mathrm P[A\circ B],
     \end{equation}
     and, by Lemma~\ref{lem:2},
     \begin{equation}
              \psi(1)=\mathbb E[F(\omega,\overline\omega)]=\mathrm P[A\cap\overline  B].
            \end{equation}
            By Lemma~\ref{lem:mainformula}, we have
            \begin{equation}
              \label{eq:12}
              \forall t\in [0,1]\quad \psi'(t)=-\frac12 \sum_{i=1}^n\mathbb E[\nabla_{ii}F(x,y)],   
            \end{equation}
            and, therefore, the desired identity would be proved if we show that 
            \begin{equation}
              -\nabla_{ii}F(x,y)=D_i F(x,y)\label{eq:13}
            \end{equation}
            for every  $i\in [n]$,  $x,y\in \{0,1\}^n $. (Recall the definition of $D_i$ from \eqref{eq:defD_i}).

            Let us fix $i\in [n]$, $x,y\in \{0,1\}^n$ and prove the expression~\eqref{eq:13}. We distinguish between two cases.

\begin{description}
\item{Case 1: $F(x_{i},y_{i})=1$ or $F(x^{i},y^{i})=0$.}

  Since $F$ is increasing, in this case all of the four terms in \eqref{eq:8} are equal. Hence, both  $\nabla_{ii} F(x,y)$ and $D_iF(x,y)$ are equal to  $0$.
  
\item{Case 2: $F(x_{i},y_{i})=0$ and $F(x^{i},y^{i})=1$.}

  In this case, \eqref{eq:13} can be rewritten as
  \begin{equation}
    \label{eq:14}
    F(x^{i},y_{i})+F(x_{i},y^{i})-1= F(x^{i},y_{i}) \cdot F(x_{i},y^{i}).
  \end{equation}
  Let $I,J\subset [n]$ be disjoint such that $(x^{i},I)$ is a witness for $A$ and $(y^{i},J)$ is a witness for $B$. Then either $I$ or $J$ does not contain $\{i\}$ so either $F(x_{i},y^{i})$ or $F(x^{i},y_{i})$ is also equal to $1$. This directly guarantees that~\eqref{eq:14} holds.

\end{description}

\end{proof}

\begin{remark}[Strong BK inequality]

By integrating~\eqref{eq:7} between $0$ and $1/2$ and writing $\xi=\omega_{1/2}$ we can deduce that for all increasing events $A,B$ 
\begin{equation}
  \label{eq:15}
  \mathrm P[A\circ B]\le \mathbb P[(A,B)\text{ occurs disjointly in }(\omega,\xi)].
\end{equation}
Since $\omega$ and $\xi$ are independent the right hand side is at most $\mathrm P[A]\mathrm P[B]$. In particular we recover the standard BK-inequality and see that~\eqref{eq:15} is a stronger form of BK-inequality. This inequality was already known; it follows easily from~\cite[Equation 2.6]{MR0814725}. Unlike Reimer's inequality, this inequality is false for general events, as remarked below Equation 2.8 in the same paper. 
    
\end{remark}

\begin{remark}[Dual Reimer inequality]
By integrating~\eqref{eq:7} between $1/2$ and $1$ and writing $\xi=\omega_{1/2}$ we get that for all increasing events $A,B$ that 

\begin{equation}
    \p[(A,B)\text{ occurs disjointly in }(\omega,\xi)]\leq P[A\cap \overline B].
\end{equation}
This inequality also holds for general events $A,B$ and is known as the dual Reimer inequality~\cite[Corollary 1.5]{MR2769191}.
\end{remark}

\subsection{FKG property of the noised pair}
\label{sec:fkg-property-noised}

We say   that a function $F:\{0,1\}^n\times\{0,1\}^n\to \mathbb R$  is increasing (resp.\@ decreasing), if it increasing (resp.\@ decreasing) when viewed as a function  on $\{0,1\}^{2n}$.

\begin{lemma}\label{lem:Holley}
Let $F,G:\{0,1\}^n\times \{0,1\}^n\rightarrow \R$ be both increasing or both decreasing. For every $t\in [0,1/2]$, we have  $$\E[F(\omega,\omega_t)G(\omega, \omega_t)]\geq \E[F(\omega, \omega_t)] \E[G(\omega, \omega_t)].$$ 
\end{lemma}

\begin{proof} Fix $t\in [0,1/2]$. Without loss of generality, we assume that both $F$ and $G$ are increasing (the decreasing case can be deduced by considering $-F$, and $-G$).
  Let $\xi\in\{-1,0,1\}^n$ be  independent of $\omega$   with iid coordinates satisfying
  \begin{equation}
        \mathbb P[\xi(i)=-1]=\mathbb P[\xi(i)=1]=t \quad \mathbb P[\xi(i)=0]=1-2t.
  \end{equation}
  Consider the configuration $\zeta\in \{0,1\}^n$ defined by
  \begin{equation}
        \zeta(i)=
    \begin{cases}
      0&\text{if }\xi(i)=-1,\\
      \omega (i) &\text{if }\xi(i) =0,\\
               1&\text{if }\xi(i)=1.
    \end{cases}
  \end{equation}
 First, we have
  \begin{equation}
    \label{eq:16}
    (\omega,\zeta)\overset{\text{law}}=(\omega,\omega_t),
  \end{equation}
  which follows for the computations
  \begin{align}
        &\mathbb P[\omega(i)=\zeta(i)=0]= \mathbb P[\omega(i)=0,\xi(i)\neq 1]= \frac12 (1-t),\\
    &\mathbb P[\omega(i)=\zeta(i)=1]= \mathbb P[\omega(i)=1,\xi(i)\neq -1]= \frac12 (1-t),\\
    &\mathbb P[\omega(i)=0,\zeta(i)=1]= \mathbb P[\omega(i)=0,\xi(i)= 1]= \frac12 t.
  \end{align}
Furthermore,  it follows from the definition that $\zeta$ is an increasing function of $(\omega, \xi)$, which implies that $F(\omega,\zeta)$ and $G(\omega,\zeta)$ are also increasing in $\omega$ and $\xi$.   Therefore, by applying the Harris-FKG inequality   to the pair $(\omega, \xi)$, we obtain
  \begin{align}
   \E[F(\omega,\zeta)G(\omega,\zeta)]\ge\E[F(\omega,\zeta)] \mathbb E[G(\omega,\zeta)],
  \end{align}
  which, together with \eqref{eq:16}, concludes the proof.

\end{proof}

The events considered in this paper are not always monotone (for instance, the events related to pivotality). Nevertheless, we will be able to use correlation inequalities for such events, provided their support satisfy some constraints. More precisely, we will use Proposition~\ref{prop:3} below, which follows from Lemma~\ref{lem:Holley} and the product structure of the probability space. This Proposition is a dynamical version of~\cite[Lemma 3]{MR0879034} and follows directly from \cite[Corollary II.2.11]{MR2108619}. We provide a proof for the completeness. 

For a subset $S\subset [n]$, we say that a function $F:\{0,1\}^n\times \{0,1\}^n\to \mathbb R$ is $S$-increasing (resp. $S$-decreasing) if for all configurations $x,y\in\{0,1\}^n$
\begin{equation}
    \forall i\in  S\quad F(x^i,y) \ge  F(x_i,y)\text{ and } F(x,y^i) \ge  F(x,y_i).
\end{equation}

\begin{proposition}
  \label{prop:3}
  Fix $t\in [0,1/2]$.   Let $A,B,S,T\subset[n]$ pairwise disjoint. Let $F,G:\{0,1\}^n\times\{0,1\}^n\to \mathbb R$ such that
  \begin{itemize}
  \item $F(\omega,\omega_t)$ is  measurable w.r.t.\@ $\{(\omega(x),\omega_t(x))\:: \: x\in A\cup S\cup T\}$, 
  \item  $G(\omega,\omega_t)$ is measurable w.r.t.\@ $\{(\omega(x),\omega_t(x))\:: \: x\in B\cup S\cup T\}$,
  \item $F,G$ are both $S$-increasing and $T$-decreasing.
  \end{itemize}
  Then for every $t\in [0,1/2]$
  \begin{equation}
        \mathbb E[F(\omega,\omega_t)G(\omega,\omega_t)]\ge\mathbb E[F(\omega,\omega_t)]\mathbb E[G(\omega,\omega_t)]
  \end{equation}
\end{proposition}

\begin{proof} Without loss of generality, we may assume $A\cup B\cup S\cup T=[n]$. Set
  \begin{equation}
\mathcal{F}=\sigma\left( (\omega(i),\omega_t(i)), \: i\in A\cup B\cup T\right). 
\end{equation}
Using that $F$ and $G$ are $S$-increasing and applying Lemma~\ref{lem:Holley} to $(\omega,\omega_t)|_S$ , we obtain the following inequality of random variables:
  \begin{equation}
\E[F(\omega,\omega_t)G(\omega,\omega_t)|\mathcal{F}]\geq \E[F(\omega,\omega_t)|\mathcal{F}]\E[G(\omega,\omega_t)|\mathcal{F}]\quad{\text{a.s.}}\label{eq:17}
\end{equation}
Let $F_1, F_2:\{0,1\}^n\times\{0,1\}^n\to \mathbb R$ such that
\begin{equation}
  \E[F(\omega,\omega_t) |  \mathcal{F}]=F_1(\omega,\omega_t)  \quad \text{and} \quad   \E[G(\omega,\omega_t) |  \mathcal{F}] =G_1(\omega,\omega_t)  
\end{equation}
almost surely. Since $F_1$ and $G_1$ are averaged versions of $F$ and $G$, they are $T$-decreasing. Setting \begin{equation}
\mathcal{G}=\sigma\big( (\omega(x),\omega_t(x)), \: x\in A\cup B\cup S\big),
\end{equation}and applying Lemma~\ref{lem:Holley} to $(\omega,\omega_t)|_T$  as above, almost surely we have
\begin{align} \E[F_1(\omega,\omega_t)G_1(\omega,\omega_t)|\mathcal{G}]&\geq \E[F_1(\omega,\omega_t)|\mathcal{G}]\E[G_1(\omega,\omega_t)|\mathcal{G}]\\&=\mathbb E[F(\omega,\omega_t)|\mathcal F\cap \mathcal G] \mathbb E[G(\omega,\omega_t)|\mathcal F\cap \mathcal G].\label{eq:18}
\end{align}
Taking the expectation in~\eqref{eq:17} and \eqref{eq:18} gives
\begin{equation}
    \mathbb E[F(\omega,\omega_t)G(\omega,\omega_t)]\ge \mathbb E\left [\mathbb E[F(\omega,\omega_t)|\mathcal F\cap \mathcal G] \mathbb E[G(\omega,\omega_t)|\mathcal F\cap \mathcal G]\right].
\end{equation}
Since $\mathcal F\cap \mathcal G $ is  generated by the configuration in $A\cup B$, the measurability hypotheses  in the first two items imply that $\mathbb E[F(\omega,\omega_t)|\mathcal F\cap \mathcal G]$ and $\mathbb E[G(\omega,\omega_t)|\mathcal F\cap \mathcal G]$ are independent. Therefore, $$ \mathbb E[F(\omega,\omega_t)G(\omega,\omega_t)]\ge\mathbb E[F(\omega,\omega_t)]\mathbb E[G(\omega,\omega_t)],$$ completing the proof.
  \end{proof}

\section{Strict inequalities for arm exponents}

\subsection{Notation and background}
\label{sec:notation-1}

We use the notation $\gtrsim$, $\lesssim$ and $\asymp$ in subsequent sections. We write $f\gtrsim g$ if there exists an absolute constant $c>0$ such that $f\geq cg$. Analogously, we write $f\lesssim g$ if there exists an absolute constant $c>0$ such that $f\leq cg$. We write $f\asymp g$ if $f\gtrsim g$ and $g \lesssim f$.

Recall that $\mathbb H$ is the set of hexagons in the plane. Given a subset $S\subset \mathbb R^2$, we define
\begin{equation}
    \mathbb H(S)=\{z\in \mathbb H: z\cap S\neq\emptyset\}.
\end{equation}
For every $n\ge k\ge 1$, we define
\begin{equation}
    \Lambda_n= [-n,n]^2\quad \text{and}\quad \Lambda_{k,n}=\Lambda_n\setminus\Lambda_k. 
  \end{equation}
  Let $A,B,S\subset \mathbb{R}^2$.  We define a path in $S$ from $A$ to $B$ as a continuous and injective map $\gamma: [0,1] \to S$ such that $\gamma(0)\in A$ and $\gamma(1)\in B$, and we identify $\gamma$ with its image. 
  
In order to view percolation events as Boolean events we make the following association. Given a colouring of the hexagons, assign to each hexagon the value $1$ if it is black and the value $0$ if it is white. More precisely, for fixed $n\ge 1$, we define a percolation configuration in $\Lambda_n$ as an element $\omega\in \{0,1\}^{\mathbb H(\Lambda_n)}$. This allows us to use the notation from Section~\ref{sec:quant-corr-ineq}, and in particular, we write $\mathrm P$ for the uniform measure in $\{0,1\}^{\mathbb H(\Lambda_n)}$. We keep the size $n$ of the box $\Lambda_n$ defining the measure $\mathrm
P$ implicit, and by convention, we fix it large enough so that all the events considered are well defined.

\paragraph{Static  events} Given some $n\ge1$, we use the term `static event' for a standard percolation event, that is,  a subset $A\subset\{0,1\}^{\mathbb H(\Lambda_n)} $. 

\paragraph{\textbf{Static one-arm and 4-arm events}}Let $n\geq m \geq 1$ and let $\sigma\in \{w, b\}$. Let ${A}_{m,n}(\sigma)$  be the event that there exists a path  connecting $\partial \Lambda_m$ to $\partial \Lambda_n$ contained in $\Lambda_{m,n}$ of colour $\sigma$.
Let $A_{m,n}(wbwb)$ be the event that there exists four paths from $\partial \Lambda_m$ to $\partial \Lambda_n$ of alternating colours black-white-black-white. We also define ${A}_{n}(\sigma)$ and $A_n(wbwb)$ similarly by requiring that the corresponding paths start from the hexagon centred at  $0$ and end at $\partial \Lambda_n$. The probabilities  of four-arm events, denoted by
\begin{equation}
\alpha_{m,n}=\mathrm P  [{A}_{m,n}(wbwb)], \quad \alpha_{n}=\mathrm P[{A}_{n}(wbwb)],
\end{equation} 
will play an important role.

\paragraph{\textbf{Dynamical events}}
For every $t\in [0,1]$, we consider  a coupled pair $(\omega,\omega_t)$ of random percolation configurations,
as in Section~\ref{sec:notation}:  $\omega$ is a random percolation configuration chosen uniformly at random and $\omega_t$ is obtained by flipping each bit of $\omega$ with probability~$t$. 

Given $n\ge 1$ and $t\in[0,1]$, we use the phrase dynamical event for an event that is measurable with respect to the pair $(\omega,\omega_t)$. To distinguish such events from static events, we denote dynamical events with calligraphic letters. For instance, many dynamical events considered in this paper are of the form
\begin{align}
    \mathcal A=\{\omega \in A,\omega_t\in B\},
\end{align}
where $A$ and $B$ are two static events.

\paragraph{\textbf{Dynamical black-white arm event}} Let $t\in [0,1]$ and $n\ge1$. Define the black-white arm event as the event that there exists a  black  arm at time $0$ and a white arm at time~$t$. This event will be central in  the proof  of Theorem~\ref{thm:polyarm} and we denote  its probability by 
\begin{equation*}
        \phi_n(t)=\p[\omega\in {A}_n(b),\omega_t\in {A}_n(w)].
      \end{equation*}

      \paragraph{Dynamical disjoint arm event } Let $t\in [0,1]$. Define the dynamical disjoint arm event as  the event that $(A_n(b),A_n(b))$ occurs disjointly in $(\omega,\omega_t)$. In words, this corresponds to the existence of a black arm at time $0$ and a black  arm at time $t$ that is disjoint from the first arm. This event will be central in the proof of Theorem~\ref{thm:monoarm} and we denote its probability by
\begin{equation*}
        \psi_n(t)=\p[({A}_n(b), {A}_n(b))\text{ occurs disjointly in $(\omega,\omega_t)$}].
      \end{equation*}

\paragraph{\textbf{Dynamical four-arm event}} For $t\in[0,1]$ and $1\le k\le n$, we say that there is a dynamical four-arm event in $\Lambda_{k,n}$ if $\omega,\omega_t\in A_{k,n}(wbwb)$.

    \subsection{Background on critical planar percolation}

\paragraph{Russo-Seymour-Welsh estimate}

For every $\lambda>0$, there exists a constant $c=c(\lambda)>0$ such that
\begin{equation}
  \label{eq:19}
  \forall n\ge 1\quad \mathrm P[\text{$\exists$ a black path from left to right in $[0,\lambda n]\times[0,n]$}]\ge c(\lambda).  
\end{equation}
By the colour-switching symmetry, the same estimate is also true for white paths. This result was first obtained by \cite{MR0488383}, and Seymour and Welsh~\cite{MR0494572}. See also \cite{MR2283880}[Chapter 3] and \cite{MR1796851}[Chapter 11]. 

\paragraph{Quasi-multiplicativity of the static four-arm event}
 
For every $1\le k\le n$, we have
\begin{equation}
  \alpha_k\alpha_{k,n}\lesssim \alpha_n\le \alpha_k\alpha_{k,n}.
\end{equation}
We refer to \cite{MR0879034,MR2438816,MR2630053} for proofs of quasi-multiplicativity.

\paragraph{Lower bound on the static four-arm event}
There exists $c>0$ such that  for every $1\le k\le n$,
  \begin{equation}
        \alpha_{k,n}\ge \left(\frac k n\right )^{2-c}. 
    \end{equation}
    This estimate is consequence of the Russo-Seymour-Welsh estimate above, and is proved by comparing the four-arm and the five-arm event, see~\cite[Lemma 6.5]{MR3468568} for details. A useful  consequence  is  the following bound.
    \begin{equation}\label{eq:20}
      \forall n\ge1 \quad \sum_{i=1}^ni\alpha_i\asymp n^2\alpha_n.
    \end{equation}
Indeed, by quasi-multiplicativity, we have for all $n\geq 1$ that

\begin{align}
   \sum_{i=1}^n i \alpha_i =\alpha_n \sum_{i=1}^n \frac{i}{\alpha_n/\alpha_i}\asymp \alpha_n\sum_{i=1}^n \frac{i}{\alpha_{i,n}},
\end{align}
and~\eqref{eq:20} follows from the two bounds $(i/n)^{2-c}\le \alpha_{i,n}\le 1$.
 
\paragraph{Behaviour  of the dynamical four-arm event}

We will use that the probability of  the dynamical four-arm  event is stable if the time $t$ is small enough. For $n\geq1$, define 
      \begin{equation}
                t_n=\min\left (\frac1{2n^2\alpha_n},\frac14\right).
      \end{equation} By \cite[ Theorem 1.4, item (i)]{MR4574826}, we have for every $ n\ge k\ge  1$, 
      \begin{equation}
        \label{eq:21}
        \forall t\le2t_n\quad \mathbb P[\omega,\omega_t\in A_{k,n}(wbwb)]\asymp\alpha_{k,n}
      \end{equation}
      and
        \begin{equation}\label{eq:22}
                \forall t\le2t_n\quad \mathbb P[\omega,\omega_t\in A_{n}(wbwb)]\asymp\alpha_{n}.
      \end{equation}

      \subsection{Two-arm versus one-arm}

      In this section, we deduce  Theorem~\ref{thm:polyarm} from Lemma~\ref{lem:pivest} below, which will be proved in Section~\ref{sec:proof-lemma}. Recall that $\phi_n(t)$ is the probability that there exists a black arm at time $0$ and a white arm at time $t$.
      
\begin{lemma}\label{lem:pivest}
Let $n\geq k \ge 1$ be such that $10k\le n$. Writing $f=\mathrm 1_{A_n(b)}$ and $g=\mathrm 1_{A_n(w)}$, we have 
\begin{equation}\label{eq:maineq}
  \forall t\in[t_k,2t_k]\quad
\sum_{x\in \mathbb H(\Lambda_{k,3k})} -\E[\nabla_xf(\omega)\nabla_xg(\omega_t) ]\gtrsim k^2\cdot \alpha_k\cdot\phi_n(t).
\end{equation}
\end{lemma}

\begin{proof}[Proof of Theorem~\ref{thm:polyarm}]
  Let $n\geq 10$. Applying Proposition~\ref{prop:1} to $f=\mathrm 1_{A_n(b)}$ and $g=\mathrm 1_{A_n(w)}$, we get
  \begin{equation}
    \label{eq:23}
    \mathrm P[A_n(b)\cap A_n(w)]=   \mathrm P[A_n(b)] \mathrm P[A_n(w)]e^{I},
  \end{equation}
  where
  \begin{align}
      I&= \frac12\int_0^{1/2} \sum_{x\in\mathbb H(\Lambda_{n})}\frac{\E[\nabla_xf(\omega)\nabla_xg(\omega_t)]}{\phi_n(t)}dt.
  \end{align}
  Let $K=\{4^{i}:0\le i\le \log_4(n/10)\}$. Using first that $\E[\nabla_xf(\omega)\nabla_xg(\omega_t)]\le 0$ (because $f$ and $g$ are respectively increasing and decreasing), we can upper bound $I$ by restricting the sum to the subset of the hexagons $x\in \cup_{k\in K}\mathbb H(\Lambda_{k,3k})$ and the integral to $t\in [t_k,2t_k]$. We obtain
  \begin{align}
    I&\overset{\phantom{Lem.~\ref{lem:pivest}}}\leq\frac12 \sum_{k\in K} \ \sum_{x\in \mathbb H(\Lambda_{k,3k})}\ \int_{t_k}^{2t_k}{\frac{-\E[\nabla_xf(\omega)\nabla_xg(\omega_t)]}{\phi_n(t)} dt}\\
     &\overset{\text{Lem.~\ref{lem:pivest}}}\lesssim - \sum_{k\in K}t_k \cdot  k^2 \cdot\alpha_{k} \\
     &\overset{\phantom{\text{Lem.~\ref{lem:pivest}}}}\lesssim - \log n.
  \end{align}
  Plugging the estimate above in \eqref{eq:23} completes the proof.
\end{proof}

\subsection{Monochromatic versus polychromatic arms}
   In this section, we deduce  Theorem~\ref{thm:monoarm} from Lemma~\ref{lem:mainlemdis} below, which will be proved in Section~\ref{sec:proof-lemma-crefl}. The proof is very similar to the proof of Theorem~\ref{thm:polyarm}.

\begin{lemma}\label{lem:mainlemdis}
  Let $n\geq k \ge 1$ be such that $10k\le n$.  Writing  $F(\eta,\xi)$ for  the indicator that $(A_n(b),A_n(b))$ occurs disjointly in $(\eta,\xi)$, we have
\begin{equation}
  \forall t\in[t_k,t_{2k}]\quad
\sum_{x\in\mathbb H( \Lambda_{k,3k})} \E[D_xF(\omega,\omega_{1-t}) ]\gtrsim k^2\cdot \alpha_k\cdot \psi_n(1-t).
\end{equation}
\end{lemma}

\begin{proof}[Proof of Theorem~\ref{thm:monoarm}]
Let $n\geq 10$. Applying Proposition~\ref{prop:2} to $A=B={A_n(b)}$  we get
  \begin{equation}
    \mathrm P[A_n(b)\circ A_n(b)]=   \mathrm P[A_n(b)\cap A_n(w)]e^{-J},\label{eq:24}
  \end{equation}
  where 
  \begin{align}
    J&=\frac12\int_0^{1} \sum_{x\in\mathbb H(\Lambda_n)}\frac{\E[D_xF(\omega,\omega_t)]}{\psi_n(t)}dt.
  \end{align}
  Let $K=\{4^{i} :  0\le i\le \log_4(n/10)\}$. Using that $\E[D_xF(\omega,\omega_t)]\ge 0$ (because $F$ is increasing), we can lower bound $J$ by restricting the sum to the subset of the hexagons $x\in \cup_{k\in K}\mathbb H(\Lambda_{k,3k})$ and the integral to $t\in [1-2t_k,1-t_k]$. We obtain
  \begin{align}
    J&\overset{\phantom{Lem.~\ref{lem:pivest}}}\geq\frac12 \sum_{k\in K} \ \sum_{x\in \mathbb H(\Lambda_{k,3k})}\ \int_{1-2t_k}^{1-t_k}{\frac{\E[D_xF(\omega,\omega_t)]}{\psi_n(t)} dt}\\
     &\overset{\text{Lem.~\ref{lem:mainlemdis}}}\gtrsim \sum_{k\in K} k^2\cdot \alpha_{k}\cdot  t_k \\
     &\overset{\phantom{Lem.~\ref{lem:pivest}}}\gtrsim \log n.
  \end{align}
  Plugging the estimate above in \eqref{eq:24} completes the proof.
\end{proof}

\subsection{Two interlaced circuits}
\label{sec:two-interl-circ}

Let $k \ge 1$ and $ t\in [t_k,2t_{k}]$.   Define the two annuli
\begin{equation}
    A_-=(-2k,-2k)+\Lambda_{3k,5k}\quad \text{and} \quad A_+=(2k,2k)+\Lambda_{3k,5k}.
\end{equation}
As illustrated on Figure~\ref{fig:1}, they overlap on two $2k\times 2k$ boxes, defined by
\begin{equation}
B_k=[-3k,-k]\times[k,3k]\quad\text{and}\quad B'_k=-B_k.
\end{equation}

  \begin{figure}[htbp]    \centering\hfill
    \begin{minipage}[c]{.3\linewidth}
    \includegraphics[width=\textwidth]{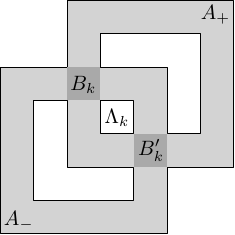}
  \end{minipage}\hfill
    \begin{minipage}[c]{.3\linewidth}
      \includegraphics[width=\textwidth]{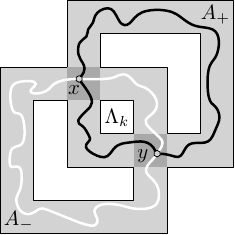}
  \end{minipage}\hfill{}
  
  \caption{Left: The two annuli $A_-$ and $A_+$ overlapping on the two boxes $B_k$ and $B_k'$. Right: Diagrammatic representation of the event $E_k(x,y)$.}
  \label{fig:1}
  \end{figure}

 Given a percolation  configuration $\eta$ and two hexagons $x\in B_k$ and $y\in B_k'$, we define $\eta^{xy}:=(\eta^x)^y$ (both $x$ and $y$ are set black) and $\eta_{xy}:=(\eta_x)_y$  (both $x$ and $y$ are set white).  For $x\in B_k$ and $y\in B_k'$, we define the static event $E_k(x,y)$  as the set of all percolation configurations $\eta$ such that
\begin{itemize}
\item in the configuration $\eta^{xy}$, there is a black circuit in $A_+$ surrounding $\Lambda_k$,
  \item in the configuration $\eta_{xy}$, there is a white circuit in $A_-$ surrounding $\Lambda_k$.
  \end{itemize} 

The event $E_k(x,y)$ is defined in terms of $\eta^{xy}$, $\eta_{xy}$, so it does not depend on the colours of $x$ and $y$.
  
  \begin{lemma}
  \label{lem:3}
  Let $k\ge 1$ and $t\in[t_k,2t_k]$. For every $x\in B_k$ and every $y\in B_k'$, we have
  \begin{equation}
        \mathbb P[\omega,\omega_t \in  E_{k}(x,y)]\gtrsim \alpha_k^2.
  \end{equation}
\end{lemma}
\begin{proof}
  Fix $k\ge 1$, $t\in[t_k,2t_k]$,  $x\in B_k$ and $y\in B_k'$.  
Write $A_+$ as the union of 4 rectangles $R_\ell,R_t,R_r$ and $R_b$ that overlap on 4 squares (see Figure~\ref{fig:2}). Similarly, write  $A_-$ as the union of 4 rectangles $S_\ell,S_t,S_r$ and $S_b$.
\begin{figure}[htbp]
  \centering
    \begin{minipage}[c]{.3\linewidth}
    \includegraphics[width=\textwidth]{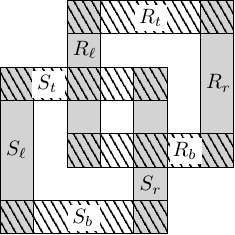}
  \end{minipage}\hfill
  \caption{Decomposition of $A_-$ and $A_+$ into rectangles}
\label{fig:2}
\end{figure}  Let $\mathsf{Piv}_x$ be the set of percolation configurations $\eta$ such that
  \begin{itemize}
  \item in $\eta_x$ there is a white path crossing $S_t$ from left to right, and
  \item in $\eta^x$, there exists a black path crossing $R_\ell$ from top  to bottom.
  \end{itemize}
  As observed in \cite[Prop. 5.9]{MR2736153}, the dynamical arm events satisfy the arm separation phenomena.  Using  this arm separation together with Russo-Seymour-Welsh estimates and Proposition~\ref{prop:3}, we obtain 
  \begin{equation}
    \label{eq:25}
    \mathbb P[\omega,\omega_t\in \mathsf{Piv}_x]\ge \mathbb P[\omega,\omega_t\in A_k({wbwb})]\overset{\eqref{eq:22}}\gtrsim \alpha_k.
  \end{equation}
  We skip the precise details in establishing the equation above and refer to the appendix where similar arguments are detailed, see Lemma~\ref{lem:5}.
  Similarly, defining  $\mathsf{Piv}_y$ as the set of percolation configurations $\eta$ such that
  \begin{itemize}
  \item in $\eta_y$ there is a white path crossing $S_r$ from top to bottom, and 
  \item in $\eta^y$, there exists a black path crossing $R_b$ from top to bottom,
  \end{itemize}
  we have 
  \begin{equation}
    \label{eq:26}
    \mathbb P[\omega,\omega_t\in \mathsf{Piv}_y]\gtrsim \alpha_k.
  \end{equation}
  Let $E$ be the static event that 
  \begin{itemize}
  \item $R_t$ and $R_r$ are both crossed by a black path in the long direction, and
  \item  $S_\ell$ and $S_b$ are both crossed by a white path in the long direction.
  \end{itemize}
  By Russo-Seymour-Welsh estimates (see~\eqref{eq:19}) and Harris-FKG inequality for percolation, we have
  \begin{equation}
        \mathrm P[E] \gtrsim 1,
  \end{equation}
  and, therefore, by monotonicity in $t$ (see Remark~\ref{rem:4}), we have
  \begin{equation}
    \label{eq:27}
    \mathbb P[\omega,\omega_t\in E]\ge \mathrm P[E]^2 \gtrsim 1.
  \end{equation}
  Finally, using Proposition~\ref{prop:3}, we deduce from \eqref{eq:25}, \eqref{eq:26} and \eqref{eq:27} that
  \begin{equation}
        \mathbb P[\omega,\omega_t\in E\cap \mathsf{Piv}_x\cap\mathsf{Piv}_y]\gtrsim \alpha_k^2.
  \end{equation}
  
\end{proof}
Let $k\ge 1$, $x\in B_k$ and $y\in B_k'$. For a colour $\sigma\in\{b,w\}$, we define the static event
  \begin{equation}
        E_k^\sigma(x,y)=\{\eta\in E_k(x,y),\eta(y)=\sigma\},
  \end{equation}
  and, for $t\in[0,1]$, the dynamical event
\begin{equation}
    \mathcal E_{k,t}(x)=\bigcup_{y\in B_k'}\{\omega\in E_k^w(x,y), \omega_t\in E_k^b(x,y)\}.
\end{equation}
Notice that the event $\mathcal E_{k,t}(x)$ implies that the hexagon $x$ is pivotal for the existence  of a white  circuit in $A_-$ at time $0$, and is pivotal for the existence of a black circuit in $A_+$ at time $t$. In the event estimated in Lemma~\ref{lem:3}, we asked (up to Russo-Seymour-Welsh constructions) that both $x$ and $y$ were pivotal, and therefore the probability was the square of the pivotality probability.  In the event  $\mathcal E_{k,t}(x)$, we require that $x$ is pivotal, but we only ask for the \emph{existence} of a pivotal in the box $B_k'$. As a consequence,  we need to pay the pivotality  probability only once (the existence of a  pivotal has a constant probability), as stated in the lemma below.

\begin{lemma}
  \label{lem:4}
  For all $k\ge 1$ and $t\in [t_k,2t_k]$, we have
  \begin{equation}
        \forall x\in B_k\quad \mathbb P[\mathcal E_{k,t}(x)]\gtrsim \alpha_k.
  \end{equation}
\end{lemma}

\begin{proof}
  Fix $k\ge 1$, $t\in[t_k,2t_k]$, and $x\in B_k$. For $y \in B_k'$, let
  \begin{equation}
        \mathcal E(y)=\{\omega\in E_k^w(x,y), \omega_t\in E_k^b(x,y)\}
  \end{equation}
  be the event appearing in the definition of   $\mathcal E_{k,t}(x)$.  Writing 
  \begin{equation}
        Y=\{y\in B_k'\: :\: \mathcal E(y)\text{ occurs}\},
  \end{equation}
  our goal is to prove that $\mathbb P[Y\ge 1]\gtrsim \alpha_k$, which we achieve by a second-moment argument. Let us start by estimating the first moment.  By definition, we have 
  \begin{equation}
        \mathcal E(y)=\{\omega,\omega_t\in E_k(x,y)\} \cap\{\omega(y)=0,\omega_t(y)=1\}.
  \end{equation}
  Since the  event  $E_k(x,y)$ does not depend on the colour of  $y$, the two events above are independent. Hence, applying Lemma~\ref{lem:3}, and using $\mathbb P[\omega(y)=0,\omega_t(y)=1]=\frac  t 2\ge \frac {t_k} 2$, we find
  \begin{equation}
       \mathbb P[  \mathcal E(y)]\gtrsim t_k\alpha_k^2.
  \end{equation}
  Since $k^2\alpha_kt_k=1$, we can estimate the first moment of $Y$ as
  \begin{equation}
  \mathbb E[Y]\gtrsim \alpha_k
\end{equation}
We now estimate the second moment
\begin{equation}
    \mathbb E[Y^2]=\sum_{y,z\in B_k'}\mathbb P[\mathcal E(y)\cap \mathcal E(z)].
\end{equation}
Let $i\in \{1,\ldots,k\}$. Let $y,z\in B_k'$ such that the $L^\infty$ distance between $y$ and $z$ satisfies $i\le \|y-z\|_\infty<i+1$. If the event $\mathcal E(y)\cap \mathcal E(z)$ occurs, then
\begin{itemize}
\item $\omega(y)=\omega(z)=0$  and $\omega_t(y)=\omega_t(z)=1$, and 
\item there is a dynamical four-arm in $\Lambda_{1,k}(x)$,  $\Lambda_{1,i/3}(y)$,  $\Lambda_{1,i/3}(z)$, and $\Lambda_{2i,k}(z)$.
  \end{itemize}
Using  independence between the events above and~\eqref{eq:21}, we obtain
\begin{equation}
 \mathbb P[ \mathcal E(y)\cap \mathcal E(z)] \le   \alpha_k \: \alpha_{i/3} ^2\: \alpha_{2i,k}\:  t^2.
\end{equation}
By quasi-multiplicativity of four-arm probability and summing over the possible distances between pairs of points in $B_k'$, we obtain
\begin{equation}
     \mathbb E[Y^2]\lesssim k^2 \alpha_k^2 \sum_{i=1}^k i\alpha_i \overset{\eqref{eq:20}}\lesssim k^4\alpha_k^3t_k^2 \lesssim\alpha_k. 
\end{equation}
Therefore, by the second moment method, we conclude
\begin{equation}
    \mathbb P[Y\ge1]\ge \frac{\mathbb E[Y]^2}{\mathbb E[Y^2]}\gtrsim \alpha_k.
\end{equation}
\end{proof}

\subsection{Proof of Lemma~\ref{lem:pivest}}
\label{sec:proof-lemma}

Fix  $n\ge k\ge  1$ such that $10k\le n$ and $t\in [t_k,2t_k$]. Let $f=\mathrm 1_{A_n(b)}$ and $g=\mathrm 1_{A_n(w)}$.
Using the notation of the previous section,  our goal is to show that
\begin{equation}\label{eq:28}
    \forall x\in B_k\quad -\mathbb E[\nabla_xf(\omega)\nabla_xg(\omega_t)]\gtrsim \phi_n(t)\alpha_k.
\end{equation}
This will conclude the proof, since $B_k\subset \Lambda_{k,3k}$ and $|B_k|\gtrsim k^2$. In order to prove the estimate above, we ``glue'' the event $\mathcal E_{k,t}(x)$ together with a well-separated version of the dynamical black-white arm event that we now define.

\paragraph{Static $k$-separated one-arm events}
Consider the rectangle  $R=[-k,k]\times[-3k,k]$ Consider the ``extended annulus'' $S=\Lambda_{7k,n}\cup ((0,6k)+\Lambda_k)$. Define the static events
\begin{align}
    A_{n}^{k\text{-sep}}(b)=\left\{
  \begin{array}[c]{c}
    \exists \text{ black path in $R$  from the hexagon centred at $0$ to the }\\
    \text{bottom of $R$ and a black path in $S$ from $\partial\Lambda_{6k}$ to   $\partial\Lambda_n$}
  \end{array}
  \right \}. 
\end{align}
and
\begin{equation}
    A_n^{k\text{-sep}}(w)=\left\{
  \begin{array}[c]{c}
    \exists \text{ white path in $-R$  from the hexagon centred at $0$ to the}\\
    \text{top of $-R$ and a white path in $-S$ from $\partial\Lambda_{6k}$ to   $\partial\Lambda_n$}
  \end{array}
  \right \} .
\end{equation}
In words, a ``$k$-separated'' arm is a path from the origin to $\partial \Lambda_n$ that is ``cut'' between  scale $3k$ and $6k$, where each of the resulting paths are required to satisfy some geometric constraints: in particular, they must  end on some specific parts of the boundary of $\Lambda_{3k}$ and $\Lambda_{6k}$ (see Figure~\ref{fig:3}).

  \begin{figure}[htbp]
  \centering 
 \hfill
    \begin{minipage}[c]{.4\linewidth}
    \includegraphics[width=\textwidth]{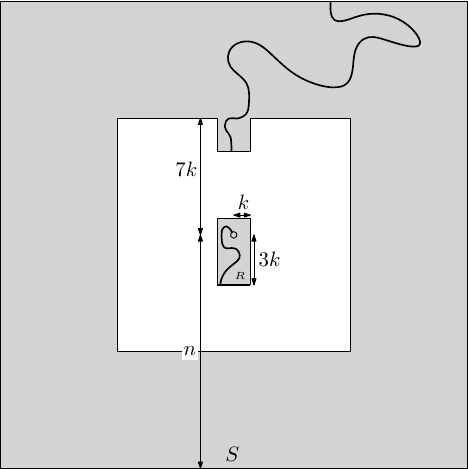}
  \end{minipage}\hfill
    \begin{minipage}[c]{.4\linewidth}
      \includegraphics[width=\textwidth]{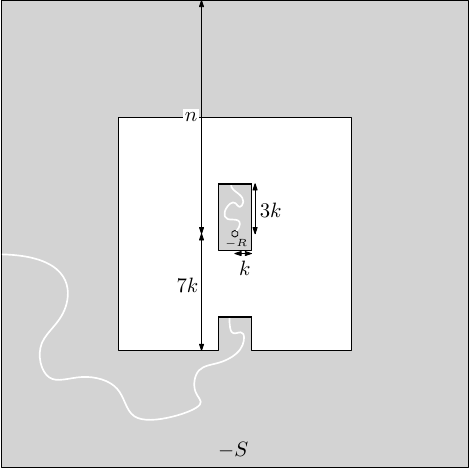}
  \end{minipage}\hfill{}
  \caption{The $k$-separated one-arm events: black (left) and white (right).}
  \label{fig:3}
\end{figure}

\paragraph{\bf{Dynamical $k$-separated black-white arm.}} Consider the dynamical event
\begin{equation}
   \mathcal A=\{\omega\in A_n^{k\text{-sep}}(b), \omega_t\in A_n^{k\text{-sep}}(w)\}.
\end{equation}
By arm separation (see Appendix \ref{sec:separation-arms}), we have
\begin{equation}
  \label{eq:29}
  \mathbb P[\mathcal A]\gtrsim \phi_n(t).
\end{equation}

  In order to prove \eqref{eq:28}, we fix $x\in B_k$. We ``glue'' the dynamical black-white arms using the two interlaced circuits of the previous section. By applying Proposition~\ref{prop:3}, we have
  \begin{equation}
    \label{eq:30}
    \mathbb P[\mathcal A\cap \mathcal E_{k,t}(x)]\ge \mathbb P[\mathcal A]\mathbb P[\mathcal E_{k,t}(x)]\gtrsim \phi_n(t)\alpha_k.
  \end{equation}
  Furthermore, by definition of the two events above, we have
  \begin{equation}
        \mathcal A\cap \mathcal E_{k,t}(x)=\bigcup_{y\in B_k'}\{\omega\in E_k^w({x,y})\cap A_n^{k\text{-sep}}(b),\omega_t\in E_k^b({x,y})\cap A_n^{k\text{-sep}}(w)\},
  \end{equation}
  The main observation is that  the static events describing the behaviour of $\omega$  (resp $\omega_t$) in the equation above imply that $x$ is pivotal for  the black (resp. white) one-arm event: as illustrated on Figure~\ref{fig:4}, for every $y\in B_k'$ and  every percolation configuration $\eta$, we have 
  \begin{equation}
        \eta\in E_{k}^w(x,y)\cap A_n^{k\text{-sep}}
        (b)\implies \nabla_xf(\eta)=1.
      \end{equation}
                    \begin{figure}[htbp]
                \centering
                \includegraphics[width=.5\textwidth]{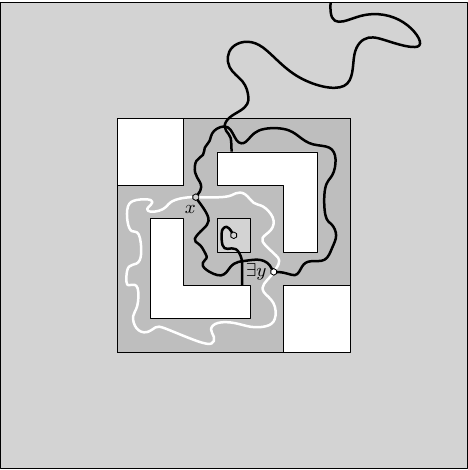}
          \caption{Diagrammatic representation of $E_{k}^w(x,y)\cap A_n^{k\text{-sep}}$ which implies that $x$ is pivotal for the black one-arm event.}
          \label{fig:4}
        \end{figure}
        Equivalently, if we consider the case when $y$ is coloured black instead of white, we have 
    \begin{equation}
        \eta\in E_k^b({x,y})\cap A_n^{k\text{-sep}}(w)\implies \nabla_xg(\eta)=-1. 
      \end{equation}
      The two displayed implications above show that 
  \begin{align}
    \mathbb P[\mathcal A\cap \mathcal E_{k,t}(x)]&\le \mathbb P[\nabla_xf(\omega)=1,\nabla_x g(\omega_t)=-1]\\
                                          &= - \mathbb E[\nabla_xf(\omega)\nabla_x g(\omega_t)],
  \end{align}
  which, together with~\eqref{eq:30},  concludes the proof of Equation~\eqref{eq:28}.

\subsection{Proof of Lemma~\ref{lem:mainlemdis}}
\label{sec:proof-lemma-crefl}

Fix  $n\ge k\ge  1$ such that $10k\le n$ and $t\in [t_k,2t_k$]. For all percolation configurations $\eta,\xi$, let $F(\eta,\xi)$ be the indicator that $(A_n(b),A_n(b))$ occurs disjointly in $(\eta,\xi)$.  Our goal is to show that
\begin{equation}\label{eq:33}
    \forall x\in B_k\quad \mathbb E[D_x F(\omega,\omega_{1-t})]\gtrsim \mathbb E[F(\omega,\omega_{1-t})]\cdot \alpha_k.
\end{equation}

 We begin with some elementary considerations in order to rewrite the quantities above in terms of $(\omega,\omega_t)$ rather than $(\omega,\omega_{1-t})$.   For all configurations $\eta, \xi$, let $\overline F(\eta,\xi)$ be the indicator that $(A_n(b),A_n(w))$ occurs disjointly in $(\omega,\xi)$. Since $A_n(b)=\overline{A_n(w)}$, we have
\begin{equation}
    \forall\eta,\xi \quad \overline F (\eta,\xi)= F(\eta,\overline \xi).
\end{equation}
Using first  $  (\omega,{\omega_{1-t}})\overset{\mathrm{law}}=(\omega,\overline{\omega_t})$  and then the notation above, we can rewrite the dynamical two-disjoint arm probability as
\begin{equation}
   \mathbb E[F(\omega,{\omega_{1-t}})]=\mathbb E[ F(\omega,\overline{\omega_{t}})]=\mathbb E[\overline F(\omega,\omega_t)].
\end{equation}
Similarly, for every $x\in B_k$, we can rewrite the left-hand side of~\eqref{eq:33} as 
\begin{align}
  \mathbb E[D_x &F(\omega,{\omega_{1-t}})]\\
  &=\mathbb E[D_x F(\omega,\overline{\omega_{t}})]\notag\\&=\mathbb E\left[\left( F(\omega^x,(\overline{\omega_t})_x)- F(\omega_x,(\overline{\omega_t})_x)\right )\cdot\left( F(\omega_x,(\overline{\omega_t})^x)- F(\omega_x,(\overline{\omega_t})_x)\right )\right]\notag\\
                                         &=\mathbb E\left[\left( \overline F(\omega^x,(\omega_t)^x)- \overline F(\omega_x,(\omega_t)^x)\right )\cdot\left( \overline F(\omega_x,(\omega_t)_x)- \overline F(\omega_x,(\omega_t)^x)\right )\right].\label{eq:34}  
\end{align}
Let us fix $x\in B_k$. In order to estimate the quantity above, we use a modification of the event $\mathcal E_{k,t}(x)$ defined in Section~\ref{sec:two-interl-circ}. Using the same notation as in Section~\ref{sec:two-interl-circ},  we define the dynamical event
\begin{equation}
   \widetilde{ \mathcal E}_{k,t}(x)=\bigcup_{y\in B_k'}\{\omega\in E_k^b(x,y), \omega_t\in E_k^w(x,y)\}.
\end{equation}
exactly as ${\mathcal E}_{k,t}(x)$ except that in the union,  we reverse the colour of $y\in B_k'$: we ask that $y$ is black at time $0$ and white at time $t$.  Since $(\omega,\omega_t)\overset{\text{law}}=(\omega_t,\omega)$, we have 
\begin{equation}
  \label{eq:35}
  \mathbb P[\widetilde{\mathcal E}_{k,t}(x)]=\mathbb P[{\mathcal E}_{k,t}(x)]\overset{\text{Lem.~\ref{lem:4}}}\gtrsim\alpha_k.
\end{equation}
Let
\begin{equation}
    \mathcal A^d =\{(A^{k\text{-sep}}(b),A^{k\text{-sep}}(w)) \text{ occurs disjointly in }(\omega,\omega_t)\}. 
\end{equation}
By arm separation (see Appendix \ref{sec:separation-arms}), we have
\begin{equation}
  \label{eq:36}
  \mathbb P[\mathcal A^d]\gtrsim \mathbb E[\overline F(\omega,\omega_t)].
\end{equation}
Applying Proposition~\ref{prop:3} and using the two displayed equations above, we find 
\begin{equation}
\mathbb P[\widetilde{\mathcal E}_{k,t}(x)\cap\mathcal A^d]\ge\mathbb P[\widetilde{\mathcal E}_{k,t}(x)]\mathbb P[\mathcal A^d]\ge\alpha_k \mathbb E[\overline F(\omega,\omega_t)].\label{eq:41}
\end{equation}
We now show that 
\begin{align}
&  \mathbb P[\widetilde{\mathcal E}_{k,t}(x)\cap\mathcal A^d] \\&\qquad\le \mathbb E\left[\left( \overline F(\omega^x,(\omega_t)^x)- \overline F(\omega_x,(\omega_t)^x)\right )\cdot\left( \overline F(\omega_x,(\omega_t)_x)- \overline F(\omega_x,(\omega_t)^x)\right )\right], \label{eq:42}
\end{align}
which, together with the displayed equation above, completes the proof of Equation~\eqref{eq:33}. Assume  $\widetilde{\mathcal E}_{k,t}(x)\cap \mathcal A^d$ occurs: by definition, there exist  two black   paths $\gamma_1, \gamma_2$ at time $0$ and two white paths $\pi_1,\pi_2$ at time $t$ such that 
\begin{itemize}
\item $\gamma_1$ is a path from the hexagon centred at 0 to the bottom of $R$ in $R$,
\item $\gamma_2$ is a path from $\Lambda_{6k}$ to $\partial \Lambda_n$ in $S$, 
\item $\pi_1$ is a path from  the hexagon centred at 0 to the top of $- R$ in $-R$,
\item $\pi_2$ is a path from $\Lambda_{6k}$ to $\partial \Lambda_n$ in $-S$,
\item  $\mathbb H(\gamma_1)\cap \mathbb H(\pi_1)=\emptyset$ and $\mathbb H(\gamma_2)\cap \mathbb H(\pi_2)=\emptyset$. 
\end{itemize}
Furthermore, there exists  $y \in B_k'$ such that 
\begin{equation}
  \omega\in E_k^b(x,y)\quad\text{and} \quad \omega_t\in E_k^w(x,y).
\end{equation}
We first claim that
\begin{equation}
  \label{eq:45}
 \overline F(\omega^x,(\omega_t)^x)=1.
\end{equation}
To see this consider a black path at time $0$ that starts with $\gamma_1$, then goes through $x$ via part of a black circuit in $A_+$, and finishes with $\gamma_2$,     and a white path at time $t$ that starts with $\pi_1$, then goes through $y$ via part of a white circuit in $A_-$, and finishes with~$\pi_2$.

We also have $\overline F(\omega_x,(\omega_t)^x)=0$. Indeed, in the configuration $\omega_x$,  any black path has to go through $y$ and in the configuration     $(\omega_t)^x$,  any white  path has to go through~$y$. Hence, there cannot exist two such disjoint paths. Therefore,

\begin{equation}
     \overline F(\omega^x,(\omega_t)^x)- \overline F(\omega_x,(\omega_t)^x)=1.
 \end{equation}
  The same argument with the role of $x$ and $y$ exchanged gives
 \begin{equation}
     \overline F(\omega_x,(\omega_t)_x)- \overline F(\omega_x,(\omega_t)^x)=1,
 \end{equation}
 and completes the proof of Equation~\eqref{eq:42}.

\appendix

\section{Appendix: Separation of arms}

\label{sec:separation-arms}
Arm separation refers to the fact that two disjoint paths starting at
$0$ will typically stay far apart: for example, if there exist a black
and a white arm from the origin to $\partial\Lambda_ n$, then one can
typically assume that they land  at distance $\gtrsim n$ of each other.
 
This was first studied by Kesten \cite{MR0879034} in the context of static percolation (see also \cite{MR2438816} and \cite{Manothesis}). For dynamical
percolation, arm separation was developed in the work of Garban, Pete
Schramm \cite[Section 5.4]{MR2736153} for events of the form $\{\omega,\omega_t\in A\}$, where $A$ is
an arm event. In our paper, we also use arm separation for events of
the form $\{\omega\in A,\omega_t\in B\}$ where $A$ and $B$ are two
(possibly different) arm events. Essentially, the standard techniques
apply in this case too.  For completeness, in this appendix we present a general argument that can be applied to prove \eqref{eq:29} and \eqref{eq:36}, largely based on \cite{duminil2021near}.

Throughout this section, we fix $t\in[0,1]$ and  drop the dependence in $t$ in our notation: we  write $\phi_k$ instead of $\phi_k(t)$ for the dynamical black-white arm probability. The goal is to prove that $\phi_k$ is comparable to the probability of the same event,  with the additional requirement that the two arms are separated at scale $k$, that is, they stay far apart at scale $k$ and are required to land on two macroscopically distant regions. In order to state this result formally, we introduce new notation. Let $k\ge 1$. Writing $R=[-k,k]\times[-3k, k]$ (as in Section \ref{sec:proof-lemma}), we define the two static events
\begin{align}
    &A_k^{\mathrm{sep}}(b)=\left\{
  \begin{array}[c]{c}
    \exists \text{ black path in $R$  from the hexagon centred at $0$  to the bottom of $R$}
  \end{array}\right \},\\
    &A_k^{\mathrm{sep}}(w)=\left\{
  \begin{array}[c]{c}
    \exists \text{ white path in $-R$  from the hexagon centred at $0$  to the top of $-R$}
  \end{array}\right \}.
\end{align}
The separated version of the dynamical black-white arm probability is defined by
\begin{equation}
    \phi_k^{\mathrm{sep}}=\mathbb P[\omega\in A_k^{\mathrm{sep}}(b),\, \omega_t\in A_k^{\mathrm{sep}}(w)].
\end{equation}
We prove the following proposition.
\begin{proposition}\label{prop:4}
  For all $k\ge 1$, we have
  \begin{equation}
       \phi_k^{\mathrm{sep}}\gtrsim \phi_k.
  \end{equation}
\end{proposition}

Proposition~\ref{prop:4} is simpler than the statement we use in the paper, but the method we present also applies to prove \eqref{eq:29} and \eqref{eq:36}. 

To prove the proposition, we rely on three lemmas. The first one asserts that separated arms can easily be extended.
\begin{lemma}\label{lem:5}
  There exists an absolute constant $c_1>0$ such that
  \begin{equation}
        \forall k\ge 1\quad \phi_{4k}^{\mathrm{sep}}\ge c_1 \phi_{k}^{\mathrm{sep}}.
  \end{equation}
\end{lemma}
The proof uses what is called a Russo-Seymour-Welsh construction, which we explain now. 
\begin{proof}
  Let $k\ge 1$.  Let $R'=[-k,k]\times[-12k,-k]$ and  $E$ be the static event that
  \begin{itemize}
  \item $R\cap R'$ is crossed horizontally by a black path, and
  \item $R'$ is crossed vertically by a black path.
  \end{itemize}
  Similarly, let $F$ be the static event that
  \begin{itemize}
  \item $-R\cap -R'$ is crossed horizontally by a white path, and
  \item $-R'$ is crossed vertically by a white path.
  \end{itemize}
  Using first that $A_k^{\mathrm{sep}}(b)\cap E\subset A_{4k}^{\mathrm{sep}}(b)$ and $A_k^{\mathrm{sep}}(w)\cap F\subset A_{4k}^{\mathrm{sep}}(w)$, and then Proposition~\ref{prop:3}, we obtain
  \begin{align}
    \phi_{4k}^{\mathrm{sep}}&\overset{\phantom{\text{Prop.}~\ref{prop:3}}}\ge \mathbb P[\omega\in    A_k^{\mathrm{sep}}(b)\cap E, \omega_t\in A_k^{\mathrm{sep}}(w)\cap F]\\
                            &\overset{\phantom{\text{Prop.}~\ref{prop:3}}}=\mathbb P[\{\omega\in    A_k^{\mathrm{sep}}(b), \omega_t\in A_k^{\mathrm{sep}}(w)\}\cap \{\omega\in E,\omega_t\in F\}]\\
   & \overset{\text{Prop.}~\ref{prop:3}}\ge   \phi_{k}^{\mathrm{sep}}\cdot\mathbb P[\omega\in E,\omega_t\in F].
  \end{align}
The result now follows from the Russo-Seymour-Welsh estimate in~\eqref{eq:19} because the events $\{\omega\in E\}$ and $\{\omega_t\in F\}$ are independent and both $\omega$ and $\omega_t$ have distribution~$\mathrm P$.
\end{proof}

Let $k\ge 1$. If the arms do not separate at scale $k$, then one should see  a certain ``bad'' event  that we now define. For $x\in \partial \Lambda_{k/2}$ and $1\le u\le v\le k/4$,  define
\begin{equation}
    A_{u,v}(x)=\left\{
  \begin{array}[c]{c}
    \exists \text{ three paths in $\Lambda_{k/2}$ from $\partial \Lambda_u(x)$ to $\partial \Lambda_v(x)$ of alternating colours}\\\text{ black-white-black or white-black-white in counterclockwise order}
  \end{array}\right \}.
\end{equation}
For $\delta\ge 1/k$, define the static event
\begin{equation}
  \label{eq:37}
  \mathsf{Bad}^\delta_k=\bigcup_{x\in\partial \Lambda_{k/2}} A_{k_1,k_2}(x)\cup {A_{3k_2,k_3}(x)}, 
\end{equation}
where $k_1=3\delta k$, $k_2=\delta^{1/3}k$,  and $ k_3=k/4$, and the dynamical event
\begin{equation}
   \mathcal B^\delta_k=\{\omega\in \mathsf{Bad}^\delta_k\}\cup \{\omega_t\in \mathsf{Bad}^\delta_k\}.
\end{equation}
We will show in Lemma~\ref{lem:7} that if we avoid this event, then the arms separate at scale~$k$. The next lemma shows that  it has  a small probability, provided $\delta$ is chosen small enough.
\begin{lemma}\label{lem:6}
  For all $k\ge 1$ and $\delta>0$ satisfying $k\ge 1/\delta$, we have
  \begin{equation}
        \mathbb P[\mathcal B_k^\delta]\lesssim \delta^{1/3}. 
  \end{equation}
\end{lemma}
\begin{proof}
  Fix $k\ge 1$ and $\delta>0$ as in the statement. Since both $\omega$ and $\omega_t$ have distribution~$\mathrm P$, it suffices to prove that
  \begin{equation}
        \mathrm P[\mathsf{Bad}_k^\delta]\lesssim \delta^{1/3}.
  \end{equation}
  Since the half-plane three-arm exponent is $2$ (see \cite[Chapter 1]{MR2523462}), we have
  \begin{equation}
    \label{eq:38}
    \forall x\in \partial \Lambda_{k/2}\ \forall 1\le u\le v\le k/4 \quad \mathrm P[A_{u,v}(x)]\lesssim\left(\frac u v\right)^2.
  \end{equation}
  Let $1\leq u, v\leq k/4$ such that $9u\le v$. Fix a sequence of points $x_1,\ldots, x_N \in \partial \Lambda_{k/2}$ in counterclockwise order such that $N\le 5k/u$, $\|x_1-x_N\|_{\infty}\le u$, and $\|x_i-x_{i+1}\|_\infty\le u$ for every $i<N$. If a three-arm event occurs around a point $x\in \partial \Lambda_{k/2} $ in the annulus $\Lambda_{u,v}(x)$ then such three arms must cross $\Lambda_{3u,v/3}$ around $x_j$, where $x_j$ is chosen such that $\|x_j-x\|_{\infty}\le u$. By union bound and the three arm estimate \eqref{eq:38}, we have
  \begin{equation}
        \mathrm P\bigg[\bigcup_{x\in \partial \Lambda_{k/2}} A_{u,v}(x) \bigg]\le  \mathrm P\bigg[\bigcup_{i=1}^N A_{3u,v/3}(x_i) \bigg]\lesssim \frac k u \cdot \left(\frac u v\right )^2=\frac{ku}{v^2}.
  \end{equation}
  Hence, by union bound again, we have
  \begin{equation}
        \mathrm P[\mathsf{Bad}^\delta_k] \le \underbrace{\mathrm P\bigg[\bigcup_{x\in \partial \Lambda_{k/2}} A_{k_1,k_2}(x)  \bigg]}_{\lesssim \delta^{1/3}}+\underbrace{\mathrm P\bigg[\bigcup_{x\in \partial \Lambda_{k/2}} A_{3k_2,k_3}(x)  \bigg]}_{\lesssim \delta^{1/3}}\lesssim \delta^{1/3},
  \end{equation}
  as desired.
\end{proof}
For $k\ge 1$, consider the dynamical event
\begin{equation}
   \mathcal A_k=\{\omega\in A_k(b),\, \omega_t\in A_k(w)\}.
\end{equation}

\begin{lemma}\label{lem:7}
  For every $\delta\in(0,1/1000)$ there exists $C(\delta)\in (0,\infty)$  such that
  \begin{equation}
       \forall k\ge \frac1\delta\qquad \mathbb P(\mathcal A_k\setminus\mathcal B^\delta_k)\le C(\delta)\cdot \phi_k^{\mathrm{sep}}.
  \end{equation}
\end{lemma}

\begin{proof}
Fix $\delta\in (0,1/1000)$ and $k\ge 1/\delta$.    Fix a sequence of points $x_1,\ldots, x_N \in \partial \Lambda_{k/2}$ in counterclockwise order such that $N\le 3/\delta$, $\|x_1-x_N\|_{\infty}\le 2\delta k$, and $\|x_i-x_{i+1}\|_{\infty}\le 2\delta k$ for every $i<N$. We start with a deterministic claim.
  \begin{claim}
    If $\eta,\xi$ are two deterministic percolation configurations satisfying $\eta\in A_k(b) \setminus \mathsf{Bad}^\delta_k$ and $\xi\in A_k(w)\setminus \mathsf{Bad}^\delta_k$, then there exist two indices $i,j$ such that $\|x_i-x_j\|_{\infty}\ge 6\delta k$ and
    \begin{itemize}
    \item in $\eta$, there is a black path from the hexagon centred at $0$ to $\Lambda_{\delta k}(x_i)$ in $\Lambda_{k/2}\setminus \Lambda_{2\delta k}(x_j)$, and
    \item in $\xi$, there is a white path from the hexagon centred at $0$ to $\Lambda_{\delta k}(x_j)$ in $\Lambda_{k/2}\setminus \Lambda_{2\delta k}(x_i)$, and
    \end{itemize}
  \end{claim}
  \begin{proof}[Proof of the claim]
    Let $k_1,k_2,k_3$ be defined as below Equation~\eqref{eq:37}. Since $\eta\in A_k(b)$, one can fix $i_0$ such that
    \begin{itemize}
    \item in $\eta$, the hexagon centred at $0$ is connected to $\Lambda_{\delta k}(x_{i_{0}})$ by a black path $\gamma_0\subset \Lambda_{k/2}$.
    \end{itemize}
    Consider the restriction of $\gamma_0$ to the annulus $\Lambda_{3k_2,k_3}(x_{i_0})$. If $\gamma_0$ was not connected to $\partial \Lambda_{k/2}$ in this annulus by a black path, then the  three-arm event $A_{3k_2,k_3}(x_{i_0})$ would occur, which would contradict $\eta\notin \mathsf{Bad}_k^\delta$. Therefore, one can also fix $i_1$ such that $\|x_{i_1}-x_{i_0}\|_{\infty}\ge 2 k_2 $ and
    \begin{itemize}
    \item in $\eta$,  the hexagon centred at $0$ is connected to $\Lambda_{\delta k}(x_{i_1})$ by a black path in $ \Lambda_{k/2}$.
    \end{itemize}
    Since $\xi\in A_k(w)$, we can fix and index $j\le N$ such that
    \begin{itemize}
    \item in $\xi$, the hexagon centred at $0$ is connected to $\Lambda_{\delta k}(x_j)$ by a white path in $\Lambda_{k/2}$.
    \end{itemize}
    Choose $i\in \{i_0,i_1\}$ such that $\|x_i-x_j\|_{\infty}\ge k_2$. At this point, in $\eta$,  we have a black path from the hexagon centred at $0$ to $\Lambda_{\delta k}(x_i)$ in $\Lambda_{k/2}$. If there  were no such path lying in $\Lambda_{k/2}\setminus \Lambda_{k_1}(x_j)$, then we would have a three-arm black-white-black in the annulus $\Lambda_{k_1,k_2}(x_j)$, which would contradict $\eta\notin \mathsf{Bad}_k^\delta$. Therefore, in $\eta$, there exists a black path from the hexagon centred at $0$ to $\Lambda_{\delta k}(x_i)$ in $\Lambda_{k/2}\setminus \Lambda_{k_1}(x_j)$. Similarly, using that $\xi\notin A_{k_1,k_2}(x_i)$, in this configuration,  the hexagon centred at $0$ is connected to $\Lambda_{\delta k}(x_j)$ by a white path in $\Lambda_{k/2}\setminus \Lambda_{k_1}(x_i)$. This concludes the proof of the claim.
    
  \end{proof}

  For $i,j$ satisfying $\|x_i-x_j\|_{\infty}\ge  6\delta k$, let $\mathcal G(i,j)$ be the event that $\omega$ satisfies the condition in the first item of the claim, and $\omega_t$ satisfies the condition in the second item of the claim. By the claim, we have
  \begin{equation}\label{eq:39}
    \mathbb P(\mathcal A_k\setminus\mathcal B^\delta_k)\le \sum_{\substack{i,j \\\|x_i-x_j\|_{\infty}\ge  6\delta k}}\mathbb P[\mathcal G(i,j)].
      \end{equation}
      Fix $i,j$ such that $\|x_i-x_j\|_{\infty}\ge 6 \delta k$. By considering two disjoint ``corridors'', one from $\Lambda_{\delta k}(x_i)$ to the bottom of $R$ in $R$, and one from  $\Lambda_{\delta k}(x_j)$ to the top of $-R$ in $-R$, one can use the  RSW estimate~\eqref{eq:19}  together with Proposition~\ref{prop:3} to prove the following. There exists $c(\delta)>0$ depending only on $\delta$  (in particular it does not depend on $k$) such that for every $i,j$ satisfying $\|x_i-x_j\|_{\infty}\ge 6\delta k$, we have   
      \begin{equation}
               \phi_k^{\mathrm{sep}} \ge c(\delta) \mathbb P[\mathcal G(i,j)].
      \end{equation}
      Plugging this estimate in \eqref{eq:39}, we obtain
      \begin{equation}
                  \mathbb P[\mathcal A_k\setminus\mathcal B^\delta_k]\le \frac{N^2}{c(\delta)} \phi_k^{\mathrm{sep}},
        \end{equation}
        which concludes the proof since $N\le 3/\delta$.
\end{proof}

\begin{proof}[Proof of Proposition~\ref{prop:4}]
  Let $c_1$ be the constant in Lemma~\ref{lem:5}. By Lemma~\ref{lem:6}, we can fix $\delta>0$ such that
  \begin{equation}
    \label{eq:40}
    \forall k\ge 1/\delta \quad \mathbb P[\mathcal B^\delta_k]\le \frac{c_1}2.
  \end{equation}
  Let $k\ge 1/\delta$. We bound each term in the decomposition
  \begin{equation}
        \phi_{k}=\mathbb P(\mathcal A_{k}\setminus\mathcal B^\delta_{k}) +\mathbb P(\mathcal A_{k}\cap\mathcal B^\delta_{k}). 
  \end{equation}
  The first term is smaller than $C(\delta)\phi_k^{\mathrm {sep}}$ by Lemma~\ref{lem:7}.  For the second term, we use independence  to get
  \begin{equation}
        \mathbb P(\mathcal A_{k}\cap\mathcal B^\delta_{k}) \le  \mathbb P(\mathcal A_{k/4}\cap\mathcal B^\delta_{k})\overset{\text{indep.}}=\phi_{k/4}\cdot \mathbb P(\mathcal B^\delta_{k}).
  \end{equation}
  Together with~\eqref{eq:40}, we obtain
  \begin{equation}
        \forall k\ge 1/\delta \quad \phi_{k}\le C(\delta)\phi_k^{\mathrm{sep}}+\frac {c_1} 2 \phi_{k/4}.
  \end{equation}
  Set $r_k=\phi_{k}/\phi_k^{\mathrm{sep}}$. Dividing by $\phi_k^{\mathrm{sep}}$  and using Lemma~\ref{lem:5} to bound  the ratio $\phi_{k/4}  /\phi_k^{\mathrm{sep}}$ by $\frac 1{c_1} r_{k/4}$, we obtain
  \begin{equation}
         \forall k\ge 1/\delta \quad r_{k}\le C(\delta) +\frac12 r_{k/4}.
  \end{equation}
  A direct induction yields that $r_k\le 2C(\delta) + C'(\delta)$ for every $k\ge 1$, where $C'(\delta)=\max\{r_\ell: \ell\le 1/\delta\}.$  
\end{proof}
    
\nocite{*}

\bibliographystyle{alpha}

\bibliography{bib}

\end{document}